\documentclass{amsart}
\usepackage{tikz,tikz-cd}
\usepackage{url}
\usepackage{graphicx}
\usepackage{graphicx}
\usepackage{color}
\usepackage{times}
\usepackage{cite}
\usepackage{enumerate,latexsym}
\usepackage{latexsym}
\usepackage{amsmath,amssymb}
\usepackage{graphicx}
\usepackage{amsthm}
\usepackage{verbatim}
\newtheorem{thm}{Theorem}[section]
\newtheorem*{theorem*}{Theorem}
\newtheorem*{acknowledgement*}{Acknowledgement}
\newtheorem{cor}[thm]{Corollary}
\newtheorem{claim}[thm]{Claim}

\newtheorem{lem}[thm]{Lemma}
\newtheorem{prop}[thm]{Proposition}
\newtheorem*{prop*}{Proposition}
\theoremstyle{definition}
\newtheorem{defn}[thm]{Definition}
\theoremstyle{remark}
\newtheorem{rem}[thm]{Remark}

\numberwithin{equation}{section}

\newcommand{\func}[1]{\ensuremath{\mathop{\mathrm{#1}}} }

\newcommand{\spt}[0]{\func{spt}}

\newcommand{\diam}[0]{\func{diam}}

\newcommand{\M}{\mathcal{M}}

\begin{document}

\title{Round spheres are Hausdorff stable under small perturbation of entropy}
\date{}
\author{Shengwen Wang}
\address{
Department of Mathematics\\
Johns Hopkins University\\
Baltimore, MD 21218, USA\\
Email:  swang@math.jhu.edu }

\maketitle

\begin{abstract}
We show that if the entropy of any closed hypersurface is close to that of a round hyper-sphere, then it is close to a round sphere in Hausdorff distance. Generalizing the result of \cite{BW1} to higher dimensions. \end{abstract}

\section{introduction}

Let $\Sigma^n\subset\mathbb R^{n+1}$ be a hypersurface, that is, a connected, smooth, properly embedded, codimension-1 submanifold. In \cite{CM}, Colding-Minicozzi introduced the entropy functional $\lambda(\Sigma)$ for such a hypersurface when studying generic singularities of the mean curvature flow. It's a natural geometric quantity that measures the complexity of a hypersurface, and is defined by
\begin{gather}
\lambda(\Sigma)=\sup_{(\mathbf y,\rho)\in\mathbb R^{n+1}\times\mathbb R}  F(\rho\Sigma+\mathbf y)
\end{gather}
where $F$ is the Gaussian area of $\Sigma$ defined by
\begin{gather}
F(\Sigma) = \int_\Sigma\frac{1}{(4\pi)^{\frac{n}{2}}} e^{-\frac{|x|^2}  {4}  }.
\end{gather}
By definition, entropy is invariant under dilations and rigid motions in $\mathbb R^{n+1}$.

The mean curvature flow in $\mathbb R^{n+1}$ is a one-parameter family of surfaces $\{\Sigma_t\}_{t\in I}$ that evolves over time $t$ by the equation
\begin{gather}
(\frac{\partial}{\partial t} \mathbf x)^\perp = \mathbf H_\Sigma (x)
\end{gather}
Here $\mathbf x$ is the position vector, $\perp$ means normal component of a vector, $\mathbf n$ is the unit normal field and $\mathbf H = -H\mathbf n=-div(\mathbf n)\mathbf n$ is the mean curvature vector.

By Huisken's monotonicity formula \cite{H}, the entropy is non-increasing under mean curvature flow. A direct computation gives that the entropy of a hyperplane is 1, minimizing entropy among all complete immersed hypsersurfaces.

It's natural to consider which closed surfaces minimize the entropy among all closed hypersurfaces, and the uniqueness and stability property of such surfaces. In \cite{BW2}, Bernstein-Wang showed that the round spheres $\mathbb S^n$ uniquely minimize the entropy (modulo dilations and rigid motions) among closed hypersurfaces in $\mathbb R^{n+1}$ for $2\leq n\leq 6$, giving an affirmative answer to a conjecture made by Colding-Ilmanen-Minicozzi-White in \cite{CIMW}. Later, Zhu \cite{Z} extended the result to all higher dimensions. In \cite{BW1}, Bernstein-Wang further showed that, for $n=2$, the round sphere is Hausdorff stable under small perturbations of entropy. In the same paper \cite{BW1}, they got an explicit relationship between the (normalized) Hausdorff distance of a surface to a round sphere and the difference between their entropies, which is analogous to the Bonnesen isoperimetric inequality for planar curves measuring the roundness of the circle by the isoperimetric defect.

In this note, we show the Hausdorff stability for the round n-sphere, generalizing the result of \cite{BW1} to closed hypersurfaces in $\mathbb R^{n+1}$.

\begin{thm}\label{thm1}
For any $\epsilon>0$, there is $\delta=\delta(\epsilon)>0$ such that, if $\Sigma$ is a closed hypersurface in $\mathbb R^{n+1}$ with entropy $\lambda(\Sigma)<\lambda(\mathbb S^n)+\delta$, then 
\begin{gather*}
\inf_{\rho>0,\mathbf y\in\mathbb R^{n+1}}\mathrm{dist}_H(\frac{1}{\rho}\Sigma-\mathbf y,\mathbb S^n)<\epsilon,
\end{gather*} 
or in other words, there exists some $\rho_0>0,\mathbf{y_0}\in\mathbb R^{n+1}$ such that 
\begin{gather*}
\mathrm{dist}_H(\Sigma,\rho_0\mathbb S^n+\mathbf{y_0})<\rho\epsilon.
\end{gather*}
\end{thm}

In \cite{BW1}, when they prove the case for $n=2$, they need to use the result in \cite{BW3} ruling out non-flat self-shrinkers with low entropy. This is not known in higher dimensions. We are not able to show that Bonnesen type isoperimetric inequality as in \cite{BW1}.

To prove the theorem above, we will choose a special kind of weak flow developed by Ilmanen \cite{I1}, and prove a uniform dependence on time of Hausdorff distance between different time slices of the flow. This uniform dependence can be viewed as a 2-sided version of Brakke's clearing out lemma (see \cite{B} 12.2) when the entropy is small, which works both backward and forward in time. As this may be of independent interest, we record it here in the setting of smooth flows:

\begin{thm}\label{thm2}
There exists $\delta(n)>0,C(n)>0,\eta(n)>0,\gamma(n)\in(0,1)$ so that: if $\{M_t^n\}$ is a a mean curvature flow of hypersurfaces in $\mathbb R^{n+1}$ that reaches the space-time point $(x_0,t_0)\in\mathbb R^{n+1}\times\mathbb R$, with entropy $\lambda(M_t)\leq\lambda(\mathbb S^n)+\delta(n)$, and $M_t\neq\emptyset$ for all $t\in(t_0-R^2,t_0+R^2)$, then for all  $0<C\rho<\gamma R$
\begin{gather*}
\mathcal H^n(B_{\rho}(x_0)\cap M_{t_0-C^2\rho^2})\geq\eta\rho^n
\end{gather*}
and
\begin{gather*}
\mathcal H^n(B_{\rho}(x_0)\cap M_{t_0+C^2\rho^2})\geq\eta\rho^n
\end{gather*}
where $\mathcal H^n$ denotes the n-dimensional  volume on hypersurfaces.
\end{thm}

\begin{rem}
This 2-sided clearing out lemma is not true for flows $\{\M_t\}$ with entropy $\lambda(\{M_t\})\geq\lambda(\mathbb S^{n-1}\times\mathbb R)$. Consider the rotational-symmetric translating "bowl" soliton, whose entropy is equal to $\lambda(\mathbb S^{n-1}\times\mathbb R)$ and is rescaled so that the speed of translation is $\frac{1}{C^2}$. For any $\gamma\in(0,1)$, by choosing $\rho>1,R>\frac{C\rho}{\gamma}$, and $(x_0,t_0)$ to be the tip of the translating "bowl" soliton, we get a counter example. However, we speculate that the theorem should still hold under the entropy bound $\lambda(\{M_t\})<\lambda(\mathbb S^{n-1}\times\mathbb R)$.
\end{rem}

\section{Notations}

\subsection{}

We denote by $B_R^{n+1}$ and $\bar B_R^{n+1}$ the open ball and closed ball in $\mathbb R^{n+1}$ with radius $R$ respectively. We omit the super-script when the dimension is clear from context.

Given two compact subsets $X, Y\subset\mathbb R^{n+1}$, the Hausdorff distance $\mathrm{dist}_H(X,Y)$ between X and Y is defined by
\begin{gather*}
\mathrm{dist}_H(X,Y)=\inf\{r>0:X\subset\cup_{x\in Y}\bar B_r(x)\:and\:Y\subset\cup_{x\in X}\bar B_r(x)\}
\end{gather*}.

For any $\rho>0$, $x_0\in\mathbb R^{n+1}$ and $\Omega\subset\mathbb R^{n+1}$, we donote by 
\begin{gather*}
\Omega-x_0=\{x\in\mathbb R^{n+1}:x+x_0\in\Omega\}\\
\rho\Omega=\{\rho x:x\in\Omega\}
\end{gather*}

\subsection{}

Following the notations in \cite{I1}, we denote by
\begin{itemize}
  \item $\mathbf M(\mathbb R^{n+1})$ = \{$\mu$: $\mu$ is a Radon measure on $\mathbb R^{n+1}$\}
  \item $\mathbf {IM}_k(\mathbb R^{n+1})$ = \{$\mu$: $\mu$ is an integer k-rectifiable Radon measure on $\mathbb R^{n+1}$\}  
  \item $\mathbf I_k(\mathbb R^{n+1})$ = \{T: T is an integeral k-current on $\mathbb R^{n+1}$\}
  \item $\mathbf {IV}_k(\mathbb R^{n+1})$ = \{V: V is an integer k-rectifiable varifold on $\mathbb R^{n+1}$\} 
\end{itemize}

$\mathbf M(\mathbb R^{n+1})$, $\mathbf {IM}_k(\mathbb R^{n+1})$ and $\mathbf {IV}_k(\mathbb R^{n+1})$ are equipped with corresponding $\text{weak}^*$ topologies. $\mathbf {I}_k(\mathbb R^{n+1})$ is equipped with the flat topology. See \cite{I1} section 1 for details of the topologies and corresponding compactness theorems. 

There are natural maps
\begin{gather*}
V_1:\mathbf {IM}_k(\mathbb R^{n+1})\rightarrow\mathbf {IV}_k(\mathbb R^{n+1})\\
V_2:\mathbf I_k(\mathbb R^{n+1})\rightarrow\mathbf {IV}_k(\mathbb R^{n+1})\\
\mu_1: \mathbf I_k(\mathbb R^{n+1})\rightarrow\mathbf {IM}_k(\mathbb R^{n+1})\\
\mu_2: \mathbf  {IV}_k(\mathbb R^{n+1})\rightarrow\mathbf {IM}_k(\mathbb R^{n+1})
\end{gather*}
Of the above maps, only $\mu_2$ is continuous. We use the following notations for convenience:
\begin{gather*}
V_1(\mu)=V(\mu)=V_\mu\\
V_2(T)=V(T)=V_T\\
\mu_1(T)=\mu(T)=\mu_T\\
\mu_2(V)=\mu(V)=\mu_V
\end{gather*}

Following definitions in \cite{W1}, an integral current $T\in\mathbf I_k(\mathbb R^{n+1})$ and an integer k-rectifiable (integral) varifold $V\in\mathbf{IV}_k(\mathbb R^{n+1})$ are said to be compatible if $V=V(T)+2W$ for some integral varifold $W\in\mathbf{IV}_k(\mathbb R^{n+1})$.

If $\Sigma^n\subset\mathbb R^{n+1}$ is a hypersurface, we denote by $\mu_{\Sigma}=\mathcal H^n\lfloor\Sigma\in\mathbf{IM}_n(\mathbb R^{n+1})$.

\subsection{}

A self-shrinker of mean curvature flow is a hypersurface $\Sigma\subset\mathbb R^{n+1}$ that satisfies the equation
\begin{gather*}
\mathbf H_\Sigma + \frac{1}{2}\mathbf x^\perp =0
\end{gather*}

It is the time $-1$ slice of a mean curvature flow that is shrinking with self-similarity
\begin{gather*}
\Sigma_t=\sqrt{-t}\Sigma_{-1},\:(t<0) \\
\Sigma_{-1}=\Sigma
\end{gather*}

The singularities of mean curvature flow are modeled on self-shrinkers by Huisken's monotonicity formula \cite{H}
\begin{gather}
\frac{d}{dt}\int_{\Sigma_t}\rho(x,t)=\int_{\Sigma_t}\rho|\mathbf H_{\Sigma_t}-\frac{1}{2t}\mathbf x^\perp|^2,\:(t<0)
\end{gather}

where $\rho(x,t)=\frac{1}{(-4\pi t)^{\frac{n}{2}}}e^{\frac{|x|^2}{4t}}$.

Self-shrinkers are critical points for the Gaussian area functional F. The entropy of a self-shrinker is equal to its Gaussian area according to computations in \cite{CM}. The associated flow for the self-shrinker has constant entropy independent of time. 

In general, the entropy of a flow is defined by 
\begin{gather*}
\lambda(\{\Sigma_t\}_{t\in I})=\sup_{t\in I}\lambda(\Sigma_t)
\end{gather*}

Important examples of self-shrinkers are generalized cylinders defined, for $0\leq k\leq n$, by
\begin{gather*}
(\sqrt{2k}\mathbb S^k)\times\mathbb R^{n-k}=\{(\mathbf x,\mathbf y)\in\mathbb R^{k+1}\times\mathbb R^{n-k}:|\mathbf x|^2=2k\}
\end{gather*}

One has
\begin{gather*}
\Lambda_k=\lambda(\mathbb S^k)=\lambda(\mathbb S^k\times\mathbb R^{n-k})
\end{gather*}

and by Stone \cite{St}
\begin{gather}
2>\Lambda_1>\frac{3}{2}>\Lambda_2>...>\Lambda_n>...\rightarrow\sqrt2
\end{gather}

\section{Weak mean curvature flows}

In this section, we gather various notions of weak mean curvature flows and prove some properties of them that will be used in this note. We mostly follow the formulations in \cite{I1}.

\subsection{}

An n-dimensional Brakke flow (Brakke motion) $\mathcal K$ in $\mathbb R^{n+1}$ is a family of Radon measures $\mathcal K = \{\mu_t\}_{t\in I}$, $\mu_t\in \mathbf {IM}_n(\mathbb R^{n+1})$, such that

(1) for a.e. t, $\mu_t=\mu(V_t)$ for some varifold $V_t\in\mathbf {IV}_n(\mathbb R^{n+1})$ so that the first variation:
\begin{gather*}
\delta V_t(X) = -\int \mathbf H(x)\cdot X d\mu_t
\end{gather*}
where $\mathbf H$ is the weak mean curvature vector field for a varifold.

(2) for any test function $f\in C_c^1(\mathbb R^{n+1}\times[a,b])$ and $f\geq0$
\begin{gather}\label{eq1}
\int f(\cdot,b)d\mu_b-\int f(\cdot,a)d\mu_a\leq \int_a^b\int (-|H|^2f+\mathbf H\cdot \nabla f+\frac{\partial f}{\partial t})d\mu_tdt
\end{gather}

A smooth mean curvature flow is automatically a Brakke motion with the inequality in (\ref{eq1}) becoming an equality. 

A Brakke flow $\{\mu_t\}_{t\in\mathbb R}$ is called eternal if $\spt\mu_t\neq\emptyset$ for all $t\in\mathbb R$.

We will restrict out attention to n-dimensional Brakke flows $\{\mu_t\}_{t\in I}$ in $\mathbb R^{n+1}$ with bounded area ratios, i.e., for which there is a $C<\infty$ so that for all $t\in I$,
\begin{gather}
\sup_{x\in\mathbb R^{n+1}}\sup_{R>0}\frac{\mu_t(B_R(x))}{R^n}\leq C
\end{gather}

Ilmanen (\cite{I2} Lemma 7) observed that the monotonicity formula of Huisken \cite{H} could be extended to the class of Brakke flows with initial data that has bounded area ratio:  
\begin{gather}
\int\rho(x,t_2)d\mu_2-\int\rho(x,t_1)d\mu_1\leq-\int_{t_1}^{t_2}\rho|\mathbf H-\frac{1}{2t}\mathbf x^\perp|^2d\mu_t(x)dt
\end{gather}
for $t_1<t_2<0$.

As a corollary, bounded area ratio at an initial time will be of bounded area ratio with the same constant in later time.

Given a flow with bounded area ratio, we define the Huisken's density $\Theta_{(x_0,t_0)}(\{\mu_t\})$ at the space-time point $(x_0,t_0)$ to be
\begin{gather*}
\Theta_{(x_0,t_0)}(\{\mu_t\})=\lim_{s\rightarrow t_0^-}\int \frac{1}{(-4\pi (s-t_0))^{\frac{n}{2}}}e^{\frac{|x-x_0|^2}{(s-t_0)} }d\mu_{s}(x)
\end{gather*}
which is upper semi-continuous by the monotonicity.

The entropy of a Brakke flow $\mathcal K=\{\mu_t\}_{t\in I}$ is defined by $\lambda(\mathcal K)=\sup_{t\in I}\lambda(\mu_t)$. It is a lower semi-continuous functional.

\begin{rem}
It is not hard to see that, for a Radon measure, bounded area ratios are equivalent to finite entropy.
\end{rem}

\begin{defn}
Let $\mathcal K_i=\{\mu_{i,t}\}_{t\geq t_0}$ be a sequence of integral Brakke flows, we say $\mathcal K_i$ converges to $\mathcal = \{\mu_t\}_{t\geq t_0}$, if 

(1) $\mu_{i,t}\rightarrow\mu_t$ for all $t\geq t_0$

(2) for a.e. $t\geq t_0$, there is a subsequence $i(k)$, depending on t, so that $V_{\mu_{i(k),t}}\rightarrow V_{\mu_t}$
\end{defn}

Convergence for flows with varying time intervals is defined analogously.

Brakke flows with uniform local mass bound have a good compactness theorem.
 
\begin{thm}(\cite{I1} section 7, cf. \cite{B} chapter 4) 

Let $\mathcal K_i=\{\mu_{i,t}\}_{t\geq t_0}$ be a sequence of n-dimensional integral Brakke flows so that for all bounded open $U\subset\mathbb R^{n+1}$,
\begin{gather*}
\sup_i\sup_{t\in[t_0,\infty)}\mu_{i,t}(U)\leq C(U)<\infty
\end{gather*}.

There is a subsequence $i(k)$ and an integral Brakke flow $\mathcal K$ so that $\mathcal K_{i(k)}\rightarrow\mathcal K$.
\end{thm}

In particular, the compactness theorem works for sequence of flows with a uniform entropy bound.

In \cite{B}, Brakke developed partial regularity theorem for Brakke flows. Later, White \cite{W3} simplified the proof for a special, but large class of Brakke flow, which include the class we use here. We will make use of a corollary of their theorem:

\begin{prop}\label{thm3}
(Proposition 3.7 of \cite{BW2}) Let $\{\mu_{i,t}\}_{t\geq t_0}$ be a sequence of integral Brakke flows converging to a limit integral Brakke flow $\{\mu_{t}\}_{t\geq t_0}$. If the limit flow is regular (smooth) in $B_R(y)\times(t_1,t_2)$, then

(1) for each $t_1<t<t_2$, $\spt(\mu_{i,t})\rightarrow\spt(\mu_t)$ in $C^\infty_{loc}(B_R(y))$

(2) given $\epsilon>0$, there is an $i_0=i_0(\epsilon,\{\mu_t\})$ so that if $i>i_0$, $\mu_{i,t}$ is regular (smooth) in $B_{R-\epsilon}(y)\times(t_1+\epsilon,t_2)$
 
\end{prop}

Denote the parabolic rescaling and translation of a Brakke flow $\mathcal K=\{\mu_t\}$ by
\begin{gather*}
D_\rho\mathcal K= \{\rho\mu_{\frac{t}{\rho^2}}\}\\
\mathcal K-(x_0,t_0)=\{\mu_{t+t_0}-x_0\}
\end{gather*}

Using Huisken's monotonicity formula, Ilmanen (\cite{I2} Lemma 8) proved that: if $\Theta_{(x_0,t_0)}>0$ (this is equivalent to $\Theta_{(x_0,t_0)}\geq1$), then there is a subsequence $\rho_i\rightarrow\infty$ such that $D_{\rho_i}(\mathcal K-(x_0,t_0))\rightarrow\mathcal{\tilde K}$. Such a limit flow $\mathcal{\tilde K}$ is called a tangent flow at $(x_0,t_0)$, and it is a backward self-shrinker for negative time.

\subsection{}

For $T\in\mathbf I_{n+1}(\mathbb R^{n+1}\times\mathbb R)$, denote $T\lfloor(\mathbb R^{n+1}\times[a,b])=T_{a\leq t\leq b}$, $\partial T_{a\leq t\leq b}=T_a-T_b$,  and $\partial T_{t\geq a}=T_a$.

A pair $(T,\mathcal K)$ is called an enhanced motion, if $T\in\mathbf I_{n+1}(\mathbb R^{n+1}\times\mathbb R)$ and $\mathcal K=\{\mu_t\}_{t\in\mathbb R}$ satisfying 

(1) $\partial T = 0$ and $\partial (T_{t\geq s})=T_s$ and $T_t\in\mathbf I_{n}(\mathbb R^{n+1})$ for each time slice t

(2) $\partial T_t = 0$ for all t and $t\mapsto T_t$ is continuous in the flat topology

(3) $\mathcal K=\{\mu_t\}_{t\in\mathbb R}$ is a Brakke motion

(4) $\mu_{T_t} \leq\mu_t$ for all t and they are compatible for a.e. t

T is the undercurrent and $\mathcal K$ is the overflow.

An enhanced motion $(T,\mathcal K)_{t\geq0}$ with initial data $T_0$ is one that condition (1) above replaced by

(1') $\partial T=T_0,\mu_{T_0}=\mu_0$, and $\partial (T_{t> s})=T_s$ and $T_t\in\mathbf I_{n}(\mathbb R^{n+1})$ for each time slice t. 

An enhanced motion in a space-time open subset $U\times I\subset\mathbb R^{n+1}\times\mathbb R$ is defined by replacing the space-time domains $\mathbb R^{n+1}$ and $\mathbb R$ in the 4 items by $U$ and $I$ repectively.

The enhanced motion $(T,\mathcal K)$ is called a matching motion if $\mu_{T_t}=\mu_t=\mu_{V_t}$ for a.e. t.  So for matching motions, we do not distinguish $\mu_{T_t},\mu_t,\mu_{V_t}$ for a.e. t. A smooth flow automatically gives rise to a matching motion.

The existence of an enhanced motion with initial data a cycle was proved by Ilmanen in \cite{I1} using an elliptic regularization procedure, and reproved by White in \cite{W1}. The continuity in flat topology (2) was not explicitly stated in \cite{I1}, but was pointed out in \cite{W1}.

There are corresponding compactness theorems for integral currents and Brakke flows with finite mass, but we cannot guarantee that the limit of matching motions is still a matching motion in general due to lower semi-continuity of the map $V_2$. A counter example is the blow-down limit of a Grim-Reaper translating soliton of (smooth) mean curvature flow is a quasi-static multiplicity 2 plane with zero undercurrent, i.e. it is not matching.

However, we can rule this out for small entropy and get a compactness theorem for matching motions with low entropy.

\begin{thm} \label{thm4}
Let $(T_i,\mathcal K_i)$ be a sequence of matching motions in $\mathbb R^{n+1}\times I$, that converge to an enhanced motion $(T,\mathcal K)$ in $\mathbb R^{n+1}\times I$. If $\lambda(\mathcal K)<2$, then the limit is also a matching motion.
\end{thm}

To prove the theorem above, we need a lemma about compatibility  of integral currents and varifolds by White \cite{W1}.

\begin{lem} \label{thm5}
(Theorem 3.6 of \cite{W1}) Suppose $V_i$ is a sequence of integer multiplicity rectifiable varifolds that converge with locally bounded first variation to an integer multiplicity rectifiable varifold, V. And $T_i$ is a sequence of integral currents such that $V_i$ and $T_i$ are compatible. If the boundaries, $\partial T_i$, converge (in the integral flat topology) to a limit integral flat chain, then there is a subsequence $i(k)$ such that $T_{i(k)}$ converge to an integral current T. Furthermore V and T must then be compatible.
\end{lem}

\begin{proof}(of Theorem \ref{thm4})

We have $\mathcal K_i=\{\mu_{i,t}\}\rightarrow\{\mu_t\}=\mathcal K$ as Brakke flows and $T_i\rightarrow T$ as currents. 

By Brakke's convergence, there is a set $S_1$ with $\mathcal L^1(S_1)=0$ (where $\mathcal L^1$ denote the Lebesgue measure), for all $t\in I \setminus S_1$, there is a subsequence $i(k)_t$, depending on t, such that, $V_{i(k),t}\rightarrow V_{\mu_t}$ with locally bounded first variation (Lemma 4.3 of \cite{W1} ).

By a slicing lemma of White (pp. 208 of \cite{W4}), there is a another set $S_2$ with $\mathcal L^1(S_2)=0$, for all $t\in I\setminus(S_1\cup S_2)$, there is a further subsequence $i(k(j))$, also depending on t, such that $T_{i(k(j)),t}\rightarrow T_t$.

Moreover, because $(T_i,\mathcal K_i)$ are matching motions, there is a set $S_3$ with $\mathcal L^1(S_3)=0$, for all $t\in I\setminus (S_1\cup S_2\cup S_3)$, $\mu_{T_{i(k(j)),t}}=\mu_{i(k(j)),t}$ and $V_{T_{i(k(j)),t}}=V_{i(k(j)),t}$. Namely, $T_{i(k(j)),t}$ and $V_{i(k(j)),t}$ are compatible.

By definition of matching motions we have $\partial T_{i(k(j)),t}=0$ for all t, so the condition of Lemma \ref{thm5} is satisfied. Thus, for each $t\in I\setminus(S_1\cup S_2\cup S_3)$, we can extract a further subsequence $i(k(j(l)))$ such that, $\lim_{i(k(j(l)))\rightarrow\infty}T_{i(k(j(l))),t}$ is compatible with $\lim_{i(k(j))\rightarrow\infty}V_{i(k(j)),t}=V_{\mu_t}$.

Since the limit of a subsequence must be same as the limit of original sequence, we have $\lim_{i(k(j(l)))\rightarrow\infty}T_{i(k(j(l))),t}=\lim_{i(k(j))\rightarrow\infty}T_{i(k(j)),t}=T_t$. 

And $T_t$ is compatible with $V_{\mu_t}$, namely
\begin{gather*}
V_{\mu_t}=V_{T_t}+2W_t
\end{gather*}

for $t\in I\setminus(S_1\cup S_2\cup S_3)$ and some $W_t\in\mathbf{IV}_n(\mathbb R^{n+1})$.

\begin{claim}\label{thm6}
For any $W\in\mathbf{IV}_n(\mathbb R^{n+1})$ and $W\neq0$, we have $\lambda(W)\geq1$
\end{claim}

\begin{proof}(of Claim \ref{thm6})

Because W is rectifiable, it is a.e. a $C^1$ submanifold with integer multiplicity. For such a point $x_0$ with $C^1$ submanifold structure, it has a tangent plane. And $\lim_{\rho\rightarrow\infty}F(\rho(\Sigma-x_0))$ will become the Euclidean density at this point, which is at least 1.
\end{proof}

Now we have for $t\in I\setminus(S_1\cup S_2\cup S_3)$
\begin{gather*}
2\lambda(W_t)\leq\lambda(V_{T_t}+2W_t)=\lambda(V_{\mu_t})\leq\lambda(\mathcal K)<2 
\end{gather*}
which forces $W_t=0$ by the Claim \ref{thm6} above.

So we have for a.e. t
\begin{gather*}
V_{\mu_t}=V_{T_t}
\end{gather*}
The limit is also a matching motion.

\end{proof}

\subsection{}

In order to get a matching motion from a generic surface, we will need another notion of set theoretic weak flow called the level-set flow. The mathematical theory of level-set flow was developed by Chen-Giga-Goto \cite{CGG} and Evans-Spruck \cite{ES1,ES2,ES3,ES4}. We follow the formulation of level-set flow of Evans-Spruck \cite{ES1}.

Let $\Gamma$ be a compact non-empty subset of $\mathbb R^{n+1}$. Select a continuous function $\mu_0$ so that $\Gamma=\{x:u_0(x)=0\}$ and there are constants $C,R>0$ so that
\begin{gather}
u_0=-C\:\:\:on\:\{x\in\mathbb R^{n+1}:|x|\geq R\}
\end{gather}
for some sufficiently large $R$. In particular, $\{u_0\geq a>-C\}$ is compact. In \cite{ES1}, Evans-Spruck established the existence and uniqueness of viscosity weak solutions to the initial value problem:
\begin{gather}\label{eq2}
\begin{cases}
u_t=\Sigma_{i,j=1}^{n+1}(\delta_{ij}-u_{x_i}u_{x_j}|Du|^{-2})u_{x_ix_j}\:\:\:on\:\mathbb R^{n+1}\times(0,\infty)\\
u=u_0\:\:\:on\:\mathbb R^{n+1}\times\{0\}
\end{cases}
\end{gather}

Setting $\Gamma_t=\{x:u(x,t)=0\}$, define $\{\Gamma_t\}_{t\geq0}$ to be the level-set flow of $\Gamma$. It is justified in \cite{ES1} that the $\{\Gamma_t\}$ is independent of the choice of $u_0$.

Level-set flow has a uniqueness property and an avoidance principle. But it may fatten up in later time, namely a level-set flow in $\mathbb R^{n+1}$ may develop some time slices that have non-zero (n+1)-dimensional Hausdorff measure. But we have the following genericity of non-fattening.

\begin{prop}\label{thm7}
(11.3 of \cite{I1}) For any closed hypersurface $\Sigma^n\subset\mathbb R^{n+1}$, and any $\epsilon>0$, and any given $k>0$, there is a small perturbation $\Sigma'$ of $\Sigma$, which is a graph $u$ over $\Sigma$ with $||u||_{C^k}<\epsilon$ and such that the level-set flow starting from $\Sigma'$ is non-fattening.
\end{prop}

A non-fattening level-set flow is a matching motion. (see \cite{I1} pp. 55)

\subsection{}

We will make use of the existence of a special kind of matching motion called the canonical boundary motion from these generic surfaces (see \cite{I1}, section 11). 

Using definitions from \cite{S}, $\partial^* E$ is the reduced boundary of $E$. If E is of locally finite perimeter, then $\mathcal H^n\lfloor\partial^* E\in\mathbf{IM}_n(\mathbb R^{n+1})$.

\begin{defn}
A $\mu\in\mathbf{IM}_n(\mathbb R^{n+1})$ is a compact boundary measure, if there is a bounded open non-empty subset $E\subset\mathbb R^{n+1}$ of locally finite perimeter so that $\spt(\mu)=\partial E$ and $\mu=\mathcal H^n\lfloor\partial^* E$. Such a set E is called the interior of $\mu$.
\end{defn}

Ilmanen synthesized both notions of weak flows and show that there is a canonical way to associate a Brakke flow to a level-set flow for a large class of initial sets. (\cite{I1}, section 11)

\begin{defn}
Given a compact boundary measure $\mu_0$ with interior $E_0$, a canonical boundary motion of $\mu_0$ is a pair $(E,\mathcal K)$ consisting of an open bounded subset $E$ of $\mathbb R^{n+1}\times\mathbb R^+$ of finite perimeter and a Brakke flow $\mathcal K=\{\mu_t\}_{t\geq t_0}$ so that:

(1) $E=\{(x,t):u(x,t)>0\}$, where u solves equation (\ref{eq2}) with $E_0=\{x:u_0(x)>0\}$ and $\partial E_0=\{x:u_0(x)=0\}$

(2) each $E_t=\{x:(x,t)\in E\}$ is of finite perimeter and $\mu_t=\mathcal H^n\lfloor\partial^*E_t$.
\end{defn}

In particular, a canonical boundary motion is a non-fattening level-set flow. Ilmanen proved the existence of canonical boundary motions (Theorem 11.4 of \cite{I1}). We need a weaker version of it

\begin{thm}\label{thm8}
(Theorem 11.4 of \cite{I1}) If $\Sigma^n\subset\mathbb R^{n+1}$ is a closed hypersurface such that the level-set flow is non-fattening, then there is a canonical boundary motion starting from $\Sigma$. In particular, it is a matching motion.
\end{thm}

The following uniqueness theorem of the flow of round sphere $\mathbb S^n\subset\mathbb R^{n+1}$ will be used in the proof of the main theorem (when applying Theorem \ref{thm3}, we need the limit flow to be regular). It was not explicitly stated in \cite{BW2}, but can be drawn as a corollary of what was proved in that paper.

\begin{thm}\label{thm9}
If $(T,\mathcal K)$ is a matching motion in $\mathbb R^{n+1}\times[-1,\infty)$, $\mathcal K=\{\mu_t\}_{t\in[-1,\infty)}$, $\lambda(\mathcal K)=\Lambda_n$. Suppose it is the limit of a sequence of compact canonical boundary motions: $\mathcal K=\lim\mathcal K_i$, with each $\mathcal K_i$ becoming extinct at $(0,0)\in\mathbb R^{n+1}\times[-1,\infty)$ , $\lambda(\mathcal K_i)\rightarrow\Lambda_n$, then $(T,\mathcal K)$ is the regular flow of a round n-sphere.
\end{thm}

\begin{proof}
By Lemma 5.1 of \cite{BW2}, the extinction time of $\mathcal K_i$ are collapsed (Definition 4.9 of \cite{BW2}). By Theorem 1.3 of the same paper \cite{BW2}, when i is large enough, the only possible tangent flow at the extinction time are round spheres.  Proposition 4.10 of \cite{BW2} implies that being collapsed is a closed condition, so the extinction time of $\mathcal K$ is also collapsed, which can only be a round sphere by the entropy bound.

Since the entropy is monotonic non-increasing under the flow, we conclude that it is constant over time and equal to that of a round sphere and thus a self-shrinker by monotonicity. Combining with the fact that it's extinct at a round-sphere tangent flow at $(0,0)$ in space-time, it must be the flow of a round shrinking n-sphere.

\end{proof}

\section{Properties of matching motions with low entropy in $\mathbb R^{n+1}$}
Some of the results in this section can be made stronger by weakening the entropy bounds in the conditions, but the versions here are enough for our purpose.

We define the set of self-shrinking measures on $\mathbb R^{n+1}$ by

\begin{gather*}
\mathcal{SM}_n=\{\mu\in\mathbf{IM}_n(\mathbb R^{n+1}): V_\mu \:\text{is stationary for Gaussian area F}\}
\end{gather*}

Denote by 
\begin{gather*}
\mathcal{CSM}_n = \{\mu\in\mathcal{SM}_n:\mu\: \text{has compact support}\}
\end{gather*} 

Further, given $\Lambda>0$, set
\begin{gather*}
\mathcal{SM}_n(\Lambda)=\{\mu\in\mathcal{SM}_n: \lambda(\mu)<\Lambda\}\\
\mathcal{CSM}_n(\Lambda)=\mathcal{CSM}_n\cap\mathcal{SM}_n(\Lambda)
\end{gather*}

According to Proposition 4.3 of \cite{BW2}, any $\mu\in\mathcal{CSM}_n(\frac{3}{2})$ is a compact boundary measure. In particular, if $n\geq3$, any $\mu\in\mathcal{CSM}_n(\Lambda_{n-1})$ is a compact boundary measure. As the corresponding results in dimension 2 is already known, we restrict ourselves to $n\geq3$ in this section.

For $\mu\in\mathcal{SM}_n(\mathbb R^{n+1})$, we call $\mathcal K=\{\mu_t\}_{t\in\mathbb R}$ an associated Brakke flow to $\mu$ if $\mu_t=\sqrt{-t}\mu$ for $t<0$. An associated matching motion to a self-shrinking measure is one whose associated overflow is an associated Brakke flow. By Theorem \ref{thm4}, any tangent flow of a matching motion with entropy bounded by 2 is an associated matching motion to a self-shrinking measure.

\begin{lem}\label{thm10}
For $\mu\in\mathcal{SM}_n(\mathbb R^{n+1})$ with $\lambda(\mu)<\infty$, let $\mathcal K$ be an associated Brakke flow to $\mu$. If there is a $y\in\mathbb R^{n+1}\setminus\{0\}$ with $\Theta_{(y,0)}\geq1$ and $\mathcal T$ is a tangent flow of $\mathcal K$ at $(y,0)$, then $\mathcal T$ splits off a line backward in time, that is $ \mathcal T_{t\leq0} = \{\tilde\mu_t\}_{t\leq0}= \{\nu_t\times\mathbb R\}_{t\leq0}$, for some $\nu_t\in\mathbf{IM}_{n-1}(\mathbb R^n)$ and $\nu_{-1}\in\mathcal SM_{n-1}(\mathbb R^n)$.
\end{lem}

\begin{proof}
For the $y\neq0$ with $\Theta_{(y,0)}\geq1$, and $\mathcal T$ being a tangent flow at $(y,0)$, there exists a sequence $\rho_i\rightarrow\infty$ such that $\rho_i(\mathcal K-(y,0))\rightarrow \mathcal T$. By the self-similarity of $\mathcal K$, we have, for any $\tau\in\mathbb R$

\begin{gather*}
\mathcal T_{t\leq0}-(\tau y,0) \\
=\lim_{i\rightarrow\infty}D_{\rho_i} (\mathcal K_{t\leq0}-(y+\frac{1}{\rho_i}\tau y,0)) \\
=\lim_{i\rightarrow\infty} D_{\rho_i(1+\frac{1}{\rho_i}\tau)}[D_{(1+\frac{1}{\rho_i}\tau)^{-1}}\mathcal K_{t\leq0} - (y,0)] \\
=\lim_{i\rightarrow\infty} D_{\rho_i(1+\frac{1}{\rho_i}\tau)} (\mathcal K_{t\leq0} -(y,0))\\
=\lim_{i\rightarrow\infty} D_{\rho_i} (\mathcal K_{t\leq0} -(y,0)) \\
=\mathcal T_{t\leq0}
\end{gather*}

where we used the fact that $\lim_{i\rightarrow\infty}(1+\frac{1}{\rho_i}\tau)=1$ and the backward self-similarity of $\mathcal K$.

Since $\tau$ is arbitrary, we conclude that the tangent flow splits off a line in the direction of $y$ backward in time.
\end{proof}

\begin{lem}\label{thm12}
If $(\tilde T,\mathcal {\tilde K})$ is an associated matching motion of an asymptotic conical self-shrinker $\Sigma^3$ that is an tangent flow of a matching motion with entropy less than $2$, then it cannot be extinct at time 0.
\end{lem} 

\begin{proof}
Because $\Sigma$ is asymptotic to a regular cone, there is $(\tilde x_0,0)$, $\tilde x_0\in\mathbb R^{n+1}\setminus\{0\}$ in the regular support of $\mu_{\tilde T_0}$, ($\Theta_{(\tilde x_0,0)}=1$). Namely, a tangent flow at $(\tilde x_0,0)$ is a multiplicity 1 plane for negative time. If 0 is the extinction time, then the tangent flow must also be 0 for positive time. We get a quasi-static multiplicity 1 plane as a tangent flow, which is not a matching motion, a contradiction to Theorem \ref{thm4}.

\end{proof}

\begin{prop}\label{thm11}
For each $n$, there exists a $\delta(n)$ such that: If $(T,\mathcal K)$ is a matching motion in $\mathbb R^{n+1}$ with $\lambda(\mathcal K)\leq\Lambda_n+\delta(n)$ that becomes extinct at time $t_0$ and $\Theta_{(x_0,t_0)}\geq1$ for some $x_0\in\mathbb R^{n+1}$, then any tangent flow at $(x_0,t_0)$ is the round n-sphere.
\end{prop}

\begin{proof}
Since $n\geq3$, if we choose $\delta(n)<\Lambda_{n-1}-\Lambda_n$, any element in $\mathcal{CSM}_n(\Lambda_{n}+\delta(n))$ is a compact boundary measure. By the results Corollary 6.5 of \cite{BW2} for dimensions $2\leq n\leq 6$ and Corollary 2.9 of \cite{Z} for all higher dimensions, we can choose some $\delta(n)<(\Lambda_{n-1}-\Lambda_n)$ so that the only element in $\mathcal{CSM}_n(\Lambda_{n}+\delta(n))$ that is a compact boundary measure is the round sphere.

For $n=3$, by Proposition 3.3 of \cite{BW4}, if $\mu\in\mathcal{SM}_3(\Lambda)$ does not have compact support, then $\mu=\mu_{\Sigma^3}$ where $\Sigma^3$ is a regular self-shrinker that is asymptotic to a regular cone (the link of the asymptotic cone is a smooth embedded hypersurface in $\mathbb S^3$).

So for $n=3$, if $t_0$ is the extinction time and $\Theta_{(x_0,t_0)}>0$, then a tangent flow at $(x_0,t_0)$ is a matching motion by Theorem \ref{thm4}, and extinct at time 0 because $(T,\mathcal K)$ is extinct at $t_0$. Combining Lemma \ref{thm12}, we conclude that a tangent flow at $(x_0,t_0)$ is round 3-sphere.

For dimension $n\geq4$, since we don't have these regularity result, we argue by induction. Suppose we know that for $k=3,...,n-1$, any k-dimensional self-shrinking matching motion that is not a sphere cannot be extinct at time 0.

If an extinction-time tangent flow of $(T,\mathcal K)$, is $\mu^n=\lim_{i\rightarrow\infty}D_{\rho_i}(\mathcal K-(x_0,t_0))$ for some $\rho_i\rightarrow\infty$, $\mu^n\in\mathcal{SM}_n(\Lambda_n+\delta(n))$, non-compact, with associated matching motion being $(\tilde T^n,\mathcal{\tilde K}^n)$, and that it's extinct at 0, we can choose $y^n_0\in\mathbb R^{n+1}-\{0\}$ such that $\Theta_{(y_0,0)}(\{\mu_{\tilde T^n_t}\})>0$, then any tangent flow at $(y_0,0)$ splits off a line backward in time by Lemma \ref{thm10}, say it is $\{\nu_t\times\mathbb R\}$ for $t\leq0$ and $\nu_{-1}\in\mathcal{SM}_{n-1}(\Lambda_n+\delta(n))\subset\mathcal{SM}_{n-1}(\Lambda_{n-1})$. Since it's a tangent flow at the extinction time, $\{\nu_t\}$ must also become extinct at time 0, and is not the (n-1)-sphere by the entropy bound, contradicting the induction hypothesis, and thus we proved the Proposition.

\end{proof}

We have the following straightforward consequence.

\begin{cor}\label{thm13}
For the same $\delta(n)$, if $\mu\in\mathcal{SM}_n(\Lambda_n+\delta(n))$ has a non-compact support, and it has associated matching motion, then this matching motion cannot be extinct at time 0.
\end{cor}

\begin{lem}\label{thm14}
Let $(T_i,\mathcal K_i=\{\mu_{i,t}\})$ be a sequence of matching motions in $\mathbb R^{n+1}$  converging to $(T,\mathcal K=\{\mu_t\})$, $\lambda(\mathcal K)\leq\Lambda_n+\delta(n)$ and
\begin{gather*}
0\in\spt(\mu_{i,0})
\end{gather*}
for all i. If it does not develop a spherical singularity for $t\in(-R,R)$, ($R>0$), then
\begin{gather}
0\in\spt(\mu_{0})
\end{gather}
\end{lem}

\begin{rem}
Without the condition that $(T,\mathcal K)$ does not develop a spherical singularity for $t\in(-R,R)$, the lemma is false. For example we can choose a sequence of regular space-time points on the shrinking sphere that converges to its extinction space-time point.
\end{rem}

\begin{proof}
By upper semi-continuity of the Huisken's density, we have $\Theta_{(0,0)}(\mu_{t})\geq1$. If $0\notin\spt(\mu_{0})$, then there is a neighborhood $U\subset \mathbb R^{n+1}$ of $0$ such that $U\cap\spt(\mu_0)=\emptyset$, and 0 is an extinction singularity for the flow $(T,\mathcal K)$ restricted to $U$. 

By Proposition \ref{thm11}, the tangent flow of $(T,\mathcal K)$ at $(0,0)$ is multiplicity-1 round sphere. By Brakke's regularity Proposition \ref{thm3}, for large enough i, $(T_i,\mathcal K_i)$ must also be a flow of a topological sphere that develops a spherical singularity before time $\frac{R}{2}$, a contradiction.
\end{proof}

\begin{prop}\label{thm15}
An ancient matching motion $(T,\mathcal K)$ in $\mathbb R^{n+1}$ with $\lambda(\mathcal K)\leq\Lambda_n+\delta(n)$, where $\delta(n)$ is given in Lemma \ref{thm9}, is either eternal or the flow of a topological sphere. 
\end{prop}

\begin{proof}

Suppose it's not eternal, it has extinction time $t_0$.

Choose $(x_0,t_0)$ such that $\Theta_{(x_0,t_0)}(\mathcal K)\geq1$. By the entropy bound and Brakke's compactness theorem, there is a blow-down sequence of flows $D_{\rho_i}(\mathcal K-(x_0,t_0))$, for some $\rho_i\rightarrow0$, converging to a limit flow $\mathcal {\tilde K}=\{\tilde\mu_t\}$. The limit is a matching motion $(\tilde T,\mathcal{\tilde K})$ by Theorem \ref{thm4} and extinct at $t=0$. Huisken's monotonicity formula \cite{H} implies that this limit flow is backwardly self-similar for $t<0$. 

By Corollary \ref{thm13}, $(\tilde T,\mathcal{\tilde K})$ must be the self-shrinking round sphere, $\tilde\mu_{-1}=\sqrt{2n}\mathbb S^n$. The convergence is multiplicity 1 by the entropy bound. So by Brakke's regularity Proposition \ref{thm4}, for large enough i, $D_{\rho_i}(\mathcal K-(x_0,t_0))$ is also the flow of a topological sphere. $\rho_i\mu_{-\frac{1}{\rho_i^2}}\rightarrow\sqrt{2n}\mathbb S^n$ in $C^\infty$ as $\rho_i\rightarrow0$.

\end{proof}

\begin{rem}
In Proposition 3.2 of \cite{BW1}, they got a stronger classification in $\mathbb R^3$ by making use of the entropy lower bound for 2-dimensional asymptotic conical self-shrinker in \cite{BW3}. That depend on a classification of genus 0 self-shrinkers by Brendle \cite{Br}, the argument of which only works in dimension 2.
\end{rem}

\begin{prop}\label{thm16}
The $\delta(n)$ can be chosen small enough so that: if $(T,\mathcal K=\{\mu_t\}_{t\in\mathbb R})$ is a matching motion in $\mathbb R^{n+1}\times\mathbb R$ with entropy $\lambda(\mathcal K)\leq\Lambda_n+\delta(n)$, that it develops a spherical singularity at $(x_0,t_0)$, then the flow is extinct at time $t_0$ at the point $x_0$.
\end{prop}

\begin{proof}
Suppose not, there is a sequence of matching motions $(T_i,\mathcal K_i=\{\mu_{i,t}\})$, with entropy $\lambda(\mathcal K_i)<\Lambda_n+\frac{1}{i}$, that develops a spherical singularity at $(x_i,t_i)$ but not extinct at time $t_i$. Without loss of generality, we can suppose $(x_i,t_i)=(0,0)$, otherwise we can do a space-time translation to make this happen.

Since the flows are not extinct at $(0,0)$, there is a point $(y_i,0)$ such that $y_i\neq0$ and $y_i\in\spt(\mu_{i,0})$.

We consider the rescaled flows $\mathcal {\tilde K}_i=D_{\frac{1}{|y_i|}}\mathcal K_i$. The new flows satisfy that $\Theta_{(0,0)}\mathcal {\tilde K}_i=\Lambda_n,\Theta_{(\frac{y_i}{|y_i|},0)}\geq1$. By Brakke's compactness theorem, we can extract a subsequence $i(k)$ so that $\mathcal {\tilde K}_{i(k)}\rightarrow\mathcal K_\infty$, and $\frac{y_i}{|y_i|}\rightarrow u$. The limit flow $\mathcal K_\infty$ is also a matching motion by Theorem \ref{thm4}. Moreover, by the upper semi-continuity of Huisken's density and the lower semi-continuity of entropy, we have
\begin{gather*}
\lambda(\mathcal K_\infty)=\Lambda_n\\
\Theta_{(0,0)}\geq\Lambda_n\\
\Theta_{(u,0)}\geq1\:\:\:\text{for some u with $|u|=1$}
\end{gather*}

But this is a contradiction, because by Huisken's monotonicity formula, for some time $t<0$, the time t slice of the flow $\mathcal K_\infty$ has entropy strictly greater than $\Lambda_n$.

\end{proof}

\section{2-sided clearing out lemma and estimate of Hausdorff distance}

We will still restrict ourselves to dimension $n\geq3$ for convenience, the 2-dimensional case was already known in \cite{BW1}. The following is Clearing out Lemma  of Brakke (\cite{B} Lemma 6.3, cf. 12.2 of \cite{I1}, Proposition 4.23 of \cite{Ec}). It can be formulated as follows

\begin{lem}\label{thm17}
(12.2 of \cite{I1}) There are constants $\eta>0,c_1>0$, depending on n such that, for any n-dimensional integral Brakke flow $\{\mu_t\}_{t\geq0},R>0$ and $(x_0,t_0)\in\mathbb R^{n+1}\times[0,\infty)$, if 
\begin{gather*}
\mu_{t_0}(B_R(x_0))\leq\eta R^n,
\end{gather*}
then
\begin{gather*}
\mu_{t_0+c_1R^2}(B_{\frac{R}{2}}(x_0))=0.
\end{gather*}
\end{lem}

\begin{prop}\label{thm18}
For $\delta(n)$ chosen satisfying Proposition \ref{thm11} and Proposition \ref{thm16}, there is $C(n)>0,\eta(n)>0$ so that: if $(T,\mathcal K=\{\mu_t\})$ is an eternal matching motion in $\mathbb R^{n+1}$, with $\lambda(\mathcal K)\leq\Lambda_n+\delta(n)$ and $0\in\spt(\mu_{0})$. Then for all $\rho>0$
\begin{gather}\label{eq3}
\mu_{-\rho^2}(B_{C\rho}(0))\geq\eta(C\rho)^n
\end{gather}
and
\begin{gather}\label{eq4}
\mu_{\rho^2}(B_{C\rho}(0))\geq\eta(C\rho)^n
\end{gather}
\end{prop}

\begin{proof}
We choose $C\geq\frac{1}{\sqrt{c_1}}$ according to Lemma \ref{thm17}. For any $\rho>0$, let $R=\frac{\rho}{\sqrt{c_1}},t_0=-\rho^2,x_0=0$ and $\eta$ smaller than one given in Lemma \ref{thm17}. We have
\begin{gather*}
\mu_{-\rho^2}(B_{C\rho}(0))\geq\eta(C\rho)^n
\end{gather*}
for otherwise by Lemma \ref{thm17} we have $\mu_{t_0+c_1R^2}(B_{\frac{R}{2}}(x_0))=\mu_0(B_{\frac{\rho}{2\sqrt{c_1}}}(0))=0$, contradicting the condition $0\in\spt(\mu_{0})$.

To prove (\ref{eq4}), we use standard blow-up argument. Suppose not, there exists a sequence $C_i\rightarrow\infty,\eta_i\rightarrow0$, satisfying $\eta_iC_i^n\rightarrow0$, a sequence of eternal matching motions $(T_i,\mathcal K_i=\{\mu_{i,t}\})$ with $\lambda(\mathcal K_i)\leq\Lambda_n+\delta(n),0\in\spt(\mu_{i,0})$, and a sequence $\rho_i>0$, such that
\begin{gather*}
\mu_{\rho^2}(B_{C_i\rho_i}(0))<\eta_i(C_i\rho_i)^n
\end{gather*}

We rescale the flows parabolically by factors $\frac{1}{\rho_i}$ to get $(\tilde T_i,\mathcal{\tilde K}_i=\{\tilde\mu_t\})$, where $\mathcal{\tilde K}_i=D_{\frac{1}{\rho_i}}\mathcal K_i$, which are also eternal flows as well. And they satisfy
\begin{gather*}
\tilde\mu_{1}(B_{C_i}(0))<\eta_i(C_i)^n\rightarrow0
\end{gather*}

By Brakke's compactness theorem, there is a subsequence $i(k)$ such that $\mathcal {\tilde K}_{i(k)}\rightarrow\mathcal{K}_\infty$, which is also a matching motion $(T_\infty,\mathcal{K}_\infty)$ by Theorem \ref{thm4}.

Since $C_{i(k)}\rightarrow\infty$ and $\tilde\mu_{1}(B_{C_{i(k)}}(0))<\eta_i(C_{i(k)})^n\rightarrow0$, the limit flow $\mathcal{K}_\infty$ must be extinct before time $t=1$, thus not eternal. By Theorem \ref{thm15}, the limit flow is the flow of a topological sphere.

The entropy bound gives the multiplicity of convergence is 1. By Brakke's regularity Proposition \ref{thm3}, the convergence is smooth and as graphs over topological spheres. For large enough $i(k)$, $\mathcal{\tilde K}_{i(k)}$ also develops a spherical singularity in finite time t. By Proposition \ref{thm16}, it must be extinct at the time a spherical singularity occur, contradicting the fact that they are eternal, and hence we proved the proposition.

\end{proof}

\begin{thm}\label{thm19}
For the same $\delta(n),C(n),\eta(n)$ chosen from Proposition \ref{thm18}, there is a $\gamma\in (0,1)$ so that: if $(T,\mathcal K=\{\mu_t\})$ is a matching motion in $\mathbb R^{n+1}$ with $\lambda(\mathcal K)\leq\Lambda_n+\delta(n$), $\{\mu_t\}$ does not develop spherical singularities for $t\in(-R^2,R^2)$, and $0\in\spt(\mu_{0})$, then for all $\rho<\gamma R$
\begin{gather}
\mu_{-\rho^2}(B_{C\rho}(0))\geq\eta(C\rho)^n
\end{gather}
and
\begin{gather}
\mu_{\rho^2}(B_{C\rho}(0))\geq\eta(C\rho)^n
\end{gather}
\end{thm}

\begin{rem}
Theorem \ref{thm2} is just a restatement of Theorem \ref{thm19} for smooth flows.
\end{rem}

\begin{proof}
Suppose not, then there is a sequence $\gamma_i\rightarrow0$, a sequence of $R_i>0$, a sequence of $\rho_i<\gamma_iR_i$, a sequence of matching motions $(T_i,\mathcal K_i=\{\mu_{i,t}\})$, with $\lambda(\mathcal K_i)\leq\Lambda_n+\delta(n)$, $\{\mu_{i,t}\}$ does not develop spherical singularities for $t\in(-R_i^2,R_i^2)$, $0\in\spt(\mu_{i,0})$, such that
\begin{gather*}
\mu_{i,-\rho_i^2}(B_{C\rho_i})<\eta(C\rho_i)^n
\end{gather*}
or
\begin{gather*}
\mu_{i,\rho_i^2}(B_{C\rho_i})<\eta(C\rho_i)^n
\end{gather*}

We rescale the flows parabolically by factors $\frac{1}{\rho_i}$ and get $(\tilde T_i,\mathcal {\tilde K}_i)$, where $\mathcal {\tilde K}_i=D_{\frac{1}{\rho_i}}\mathcal K_i$, with $\lambda(\mathcal{\tilde K}_i)\leq\Lambda_n+\delta(n)$,  satisfying
\begin{gather*}
\mu_{i,-1}(B_{C})<\eta(C)^n
\end{gather*}
or
\begin{gather*}
\mu_{i,1}(B_{C})<\eta(C)^n
\end{gather*}

Because of rescaling, we also have $\{\tilde\mu_{i,t}\}$ does not develop spherical singularities for $t\in(-\frac{R_i^2}{\rho_i^2},\frac{R_i^2}{\rho_i^2})$, where $\frac{R_i^2}{\rho_i^2}\geq\frac{1}{\gamma_i^2}\rightarrow\infty$.

By Brakke's compactness theorem, there is a subsequence $i(k)$ such that $\mathcal {\tilde K}_{i(k)}\rightarrow\mathcal K_\infty$, which is also a matching motion $(T_\infty,\mathcal K_\infty=\{\mu_{\infty,t}\})$ by Theorem \ref{thm4}. $\mathcal K_\infty$ is either eternal or the flow of a topological sphere by Proposition \ref{thm15}, and $0\in\spt(\mu_{\infty,0})$ by Lemma \ref{thm14}. And satisfying
\begin{gather*}
\mu_{\infty,-1}(B_{C})<\eta(C)^n
\end{gather*}
or
\begin{gather*}
\mu_{\infty,1}(B_{C})<\eta(C)^n
\end{gather*}
If $\mathcal K_\infty$ is eternal, by choosing $\rho=1$, we get a contradiction to Proposition \ref{thm18}.

If $\mathcal K_\infty$ is the flow of a topological sphere, say $\mu_{\infty,t}$ is a topological sphere that is extinct at time $\tilde t$, then By Brakke's regularity Theorem Proposition \ref{thm3}, for large enough $i(k)$, the flow also develops a spherical singularity before time $2\tilde t$, contradicting the fact that it does not develop spherical singularities for $t\in(-\frac{R_{i(k)}^2}{\rho_{i(k)}^2},\frac{R_{i(k)}^2}{\rho_{i(k)}^2})$ where $\frac{R_{i(k)}^2}{\rho_{i(k)}^2}\rightarrow\infty$.

\end{proof}

\section{Proof of Theorem \ref{thm1}}

We are now ready to prove the main theorem. First we need the following consequence of what's proved in the previous section.

\begin{lem}\label{thm20}
For $n\geq3$ and $\delta(n)$ chosen as in the previous section, there is a $C(n,\gamma)$: If $(T,\mathcal K=\{\mu_t\})$ is a matching motion in $\mathbb R^{n+1}\times[0,t_0]$ with $\lambda(\mathcal K)\leq\Lambda_n+\delta(n)$, and $\mu_{t}$ does not develops spherical singularities for $t\in(0,t_0)$, then for any $0\leq t_1<t_2\leq t_0$, we have
\begin{gather}
\mathrm{dist}_H(\spt(\mu_{t_1}),\spt(\mu_{t_2}))\leq C\sqrt{t_2-t_1}
\end{gather}
\end{lem}

\begin{proof}

Theorem \ref{thm19} gives us a $C>0,\gamma\in(0,1)$ such that for any $t\in(0,t_0)$, if $0<\tau<\min(\gamma t,\gamma(t_0-t))$, then
\begin{gather}\label{eq5}
\text{$\spt(\mu_{t+\tau})$ is in the $C\sqrt\tau$ neighborhood of $\spt(\mu_{t})$}
\end{gather}
and
\begin{gather}\label{eq6}
\text{$\spt(\mu_{t-\tau})$ is in the $C\sqrt\tau$ neighborhood of $\spt(\mu_{t})$}
\end{gather}

By replacing t with $t+\tau$, (\ref{eq6}) also gives us
\begin{gather}
\text{$\spt(\mu_{t})$ is in the $C\sqrt\tau$ neighborhood of $\spt(\mu_{t+\tau})$}
\end{gather}

Namely
\begin{gather}
\mathrm{dist}_H(\spt(\mu_{t}),\spt(\mu_{t+\tau}))<C\sqrt{\tau}
\end{gather}

Now for any $0< t_1<t_2<t_0$, we can choose $\eta_1,\eta_2$, depending on $\gamma$, such that $\frac{1}{2}\gamma<\eta_1,\eta_2<\gamma$ and $\frac{t_1+t_2}{2}=(1+\eta_1)^{k_1} t_1$ and $(t_0-t_2)(1+\eta_2)^{k_2} =(t_0-\frac{t_1+t_2}{2})$ for some $k_1,k_2\in\mathbb N$. 

For $i_1=0,...,k_1-1$, we have
\begin{gather*}
\mathrm{dist}_H(\spt(\mu_{t_1(1+\eta_1)^{i_1}}),\spt(\mu_{t_1(1+\eta_1)^{i_1+1}})  )\\
\leq C \sqrt{t_1(1+\eta_1)^{i_1+1}-t_1(1+\eta_1)^{i_1}} \\
= C\sqrt{t_1\eta_1}(\sqrt{1+\eta_1})^{i_1},
\end{gather*}
so
\begin{gather*}
\mathrm{dist}_H(\spt(\mu_{\frac{t_1+t_2}{2}}),\spt(\mu_{t_1}))\\
\leq\Sigma_{i_1=0}^{k_1-1}\mathrm{dist}_H(\spt(\mu_{t_1(1+\eta_1)^{i_1}}),\spt(\mu_{t_1(1+\eta_1)^{i_1+1}})  )\\
\leq\Sigma_{i_1=0}^{k_1-1}C\sqrt{t_1\eta_1}(\sqrt{1+\eta_1})^{i_1}\\
= C\sqrt{t_1\eta_1}\frac{1-(\sqrt{1+\eta_1})^{k_1}}{1-\sqrt{1+\eta_1}}\\
= C\frac{\sqrt{\eta_1}}{\sqrt{1+\eta_1}-1}(\sqrt{t_1(1+\eta_1)^{k_1}}-\sqrt{t_1}) \\
\leq C\frac{\sqrt{\eta_1}}{\sqrt{1+\eta_1}-1}(\sqrt{t_1(1+\eta_1)^{k_1}-t_1})\\
=C(\eta_1)(\sqrt{\frac{t_1+t_2}{2}-t_1})\\
\leq C(\gamma)(\sqrt{\frac{t_1+t_2}{2}-t_1}).
\end{gather*}.

Similarly for $i_2=0,...,k_2-1$, we have
\begin{gather*}
\mathrm{dist}_H(\spt(\mu_{t_0-(t_0-t_2)(1+\eta_2)^{i_2}}),\spt(\mu_{t_0-(t_0-t_2)(1+\eta_2)^{i_2+1}})  )\\
\leq C \sqrt{(t_0-t_2)(1+\eta_2)^{i_2+1}-(t_0-t_2)(1+\eta_2)^{i_2}} \\
= C\sqrt{(t_0-t_2)\eta_2}(\sqrt{1+\eta_2})^{i_2},
\end{gather*}
so
\begin{gather*}
\mathrm{dist}_H(\spt(\mu_{t_2}),\spt(\mu_{\frac{t_1+t_2}{2}}))\\
\leq\Sigma_{i_2=0}^{k_2-1}\mathrm{dist}_H( \spt(\mu_{t_0-(t_0-t_2)^{i_2}}) , \spt(\mu_{t_0-(t_0-t_2)^{i_2+1}})  )\\
\leq\Sigma_{i_2=0}^{k_2-1}C\sqrt{(t_0-t_2)\eta_2}(\sqrt{1+\eta_2}^{i_2})\\
= C\sqrt{(t_0-t_2)\eta_2}\frac{1-\sqrt{1+\eta_2}^{k_2}}{1-\sqrt{1+\eta_2}}\\
=  C \frac{\sqrt{\eta_2}}{\sqrt{1+\eta_2}-1} (\sqrt{(t_0-t_2)(1+\eta_2)^{k_2}}-\sqrt{(t_0-t_2)(1+\eta_2)}) \\
\leq C \frac{\sqrt{\eta_2}}{\sqrt{1+\eta_2}-1} (\sqrt{(t_0-t_2)(1+\eta_2)^{k_2}-(t_0-t_2)(1+\eta_2)}) \\
=  C \frac{\sqrt{\eta_2}}{\sqrt{1+\eta_2}-1} (\sqrt{[t_0-(t_0-t_2)]-[t_0-(t_0-t_2)(1+\eta_2)^{k_2}]}) \\
= C(\eta_2)(\sqrt{t_2-\frac{t_1+t_2}{2}})\\
\leq C(\gamma)(\sqrt{t_2-\frac{t_1+t_2}{2}}).
\end{gather*}

Now by triangle inequality

\begin{gather*}
\mathrm{dist}_H(\spt(\mu_{t_2}),\spt(\mu_{t_1}))\\
\leq C(\gamma)(\sqrt{\frac{t_1+t_2}{2}-t_1})+C(\gamma)(\sqrt{t_2-\frac{t_1+t_2}{2}})\\
\leq C\sqrt{t_2-t_1}
\end{gather*}

For the case $t_1=0$ or $t_2=t_0$, since $C$ is independent of $t_1,t_2$, we can take limits and thus proved the lemma.

\end{proof}

\begin{prop}\label{thm21}
The $\delta(n)$ can be chosen small enough so that, if $(T,\mathcal K=\{\mu_t\}_{t\geq0})$ is a matching motion in $\mathbb R^{n+1}$ with initial data $\mu_0=\mu_{\Sigma_0}$ being a closed hypersurface, $\lambda(\mathcal K)<\Lambda_n+\delta(n)$, then if $\{\mu_t\}$ develops a spherical singularity at space-time point $(x_0,t_0)\in\mathbb R^{n+1}\times(0,\infty)$, then the flow is extinct at time $t_0$ at a spherical singularity at $x_0$.
\end{prop}

\begin{proof}
Without loss of generality, we can assume that $x_0=0,t_0=1$, for otherwise we can do a parabolic  translation and dilation to make this happen. The proof is by contradiction

\begin{claim}\label{thm22}
If the flow is not extinct at the time $t=1$ when it first develops a spherical singularity. Then there is a point $y_0\in B_{4C}(0),y_0\neq0$ with Huisken's density $\Theta_{(y_0,t_0)}\geq1$, where $C$ is the universal constant from the previous Lemma \ref{thm20}.
\end{claim}

\begin{proof}(of Claim \ref{thm22})
Without the assumption that the point $y_0\in B_{4C}(0)$, the existence is straightforward since the flow is not extinct yet at time $t=1$. 

Case 1: If $\spt(\mu_0)\subset B_{3C}$, by the Hausdorff estimate Lemma \ref{thm20}, we have 
\begin{gather*}\mathrm{dist}_H(\spt(\mu_0),\spt(\mu_1))\leq C
\end{gather*} 
namely, there is a point $y_0\in\spt(\mu_1)$ such that
\begin{gather*}
y_0\in B_{3C+C}(0)=B_{4C}(0)
\end{gather*}
We can choose $y_0\neq0$ because the flow is not extinct at $t=1$ and $(0,1)$ is a spherical singularity.

Case 2: If $\spt(\mu_0)=\mu_{\Sigma_0}$ is not contained in $B_{3C}$, by the connectedness of $\Sigma_0$, there is a point $z_0\in\Sigma_0\cap(B_{3C}\setminus B_{2C})$. Again by the Hausdorff distance estimate Lemma \ref{thm20}, we have can find a point $y_0\in\spt(\mu_1)$ such that 
\begin{gather*}
y_0\in (B_{3C+C}\setminus B_{2C-C})=(B_{4C}\setminus B_C)
\end{gather*}

\end{proof}

If the Proposition were false, there would be a sequence of matching motions $(T_i,\mathcal K_i=\{\mu_{i,t}\}_{t\geq0})$, each of which satisfying that $\mu_{i,0}$ is a closed hypersurface, $\lambda(\mathcal K_i)\leq\Lambda_n+\frac{1}{i}$. They all develops a spherical singularity at $(0,1)\in\mathbb R^{n+1}\times(0,\infty)$ and has a point $y_i\in\spt(\mu_{i,1})\cap B_{4C},y_i\neq0$ such that $\Theta_{(y_i,1)}\geq1$. 

We use Brakke's compactness theorem to extract a sub-sequential limit flow $\mathcal K_\infty=\{\mu_{\infty,t}\}_{t\geq0}$, which is also a matching motion by Theorem \ref{thm4}. By the lower semi-continuity of entropy, $\lambda(\mathcal K_\infty)=\Lambda_n$. By the upper semi-continuity of the Huisken's density, $\Theta_{(0,1)}\geq\Lambda_n$ and $\Theta_{(y_\infty,1)}\geq1$ for some $y_\infty\neq0$.

Now this is a contradiction because, by Huisken's monotonicity formula, at an earlier time than $t=1$, the entropy is strictly greater than $\Lambda_n$. Thus we proved the Proposition.
\end{proof}

\begin{proof}
(of Theorem \ref{thm1})

As in \cite{BW1}, we argue by contradiction. Since we cannot rule out the formation of non-compact singularities with low entropy as in \cite{BW3}, we must work in the setting of matching motions instead of smooth flows. According to Proposition \ref{thm21}, for flows starting from a closed hypersurface, the only time when the flow develops a spherical singularity is when it's extinct.

Fix the dimension n, suppose for some $1>\epsilon>0$ there are connected closed hypersurfaces $\Sigma_i\subset\mathbb R^{n+1}$ with $\lambda(\Sigma_i)\leq\Lambda_n+\frac{1}{2i}$ and so that $\mathrm{dist}_H(\rho\mathbb S^n+y,\Sigma_i)>\rho\epsilon>0$ for any $\rho>0,y\in\mathbb R^{n+1}$.

\begin{claim}\label{thm23}("Triangle" inequality for the scale-invariant Hausdorff distance)

For each such i, we can find a small graphical perturbation $\tilde\Sigma_i$, also connected, such that the level-set flow of $\tilde\Sigma_i$ is non-fattening and such that
\begin{gather*}
\mathrm{dist}_H(\rho\mathbb S^n+\mathbf y,\tilde\Sigma_i)\geq\frac{\rho\epsilon}{2}>0
\end{gather*}
for any $\rho>0, \mathbf y\in\mathbb R^{n+1}$, and
\begin{gather*}
\lambda(\tilde\Sigma_i)\leq\Lambda_n+\frac{1}{i}
\end{gather*}
\end{claim}

\begin{proof}(of Claim \ref{thm23})

By Proposition \ref{thm7}, we can choose $\tilde\Sigma_i$ being non-fattening and such that $\mathrm{dist}_H(\Sigma_i,\tilde\Sigma_i)$ is arbitrary small because the Hausdorff distance is bounded by the $C^0$ graphical norm.

First, for $\rho<\frac{1}{16}\diam(\Sigma_i)$, if $\tilde\Sigma_i$ is chosen so that $\mathrm{dist}_H(\Sigma_i,\tilde\Sigma_i)<\frac{1}{16}\diam(\Sigma_i)$, we already have $\mathrm{dist}_H(\rho\mathbb S^n+\mathbf y,\tilde\Sigma_i)\geq\frac{\rho\epsilon}{2}$ for any $\mathbf y\in\mathbb R^{n+1}$. This is because $\diam(\tilde\Sigma_i)>\frac{14}{16}\diam(\Sigma_i)>14\rho$, so cannot lie in the $\frac{\rho\epsilon}{2}$ neighborhood of any round n-sphere of radius $\rho$ in $\mathbb R^{n+1}$. Namely, for any $0<\rho<\frac{1}{16}\diam(\Sigma_i),\mathbf y\in\mathbb R^{n+1}$,
\begin{gather*}
\mathrm{dist}_H(\tilde\Sigma_i,\rho\mathbb S^n+\mathbf y)\geq\frac{\rho\epsilon}{2}
\end{gather*}

Next, for $\rho\geq\frac{1}{16}\diam(\Sigma_i)$, we will choose $\tilde\Sigma_i$ so that $\mathrm{dist}_H(\tilde\Sigma_i,\Sigma_i)<\frac{1}{32}\diam(\Sigma_i)\epsilon<\frac{\rho\epsilon}{2}$. Then for any $\mathbf y$,
\begin{gather*}
\mathrm{dist}_H(\tilde\Sigma_i,\rho\mathbb S^n+\mathbf y)\\
\geq \mathrm{dist}_H(\Sigma_i,\rho\mathbb S^n+\mathbf y)-\mathrm{dist}_H(\tilde\Sigma_i,\Sigma_i) \\
\geq \rho\epsilon-\frac{\rho\epsilon}{2}\\
=\frac{\rho\epsilon}{2}
\end{gather*}

By the lower semi-continuity of entropy, there is a $\delta_i>0$ such that if the $\tilde\Sigma_i$ chose as a graph $u_i$ over $\Sigma_i$ with $||u_i||_{C^0(\Sigma_i)}<\delta$, then $\lambda(\tilde\Sigma_i)<\lambda(\Sigma_i)+\frac{1}{2i}\leq\Lambda_n+\frac{1}{i}$.

Thus, we can choose $\tilde\Sigma_i$ according to Proposition \ref{thm7} so that it's a graph $u_i$ over $\Sigma_i$ with $||u_i||_{C^0(\Sigma_i)}<\min(\frac{1}{16}\diam(\Sigma_i),\frac{1}{32}\diam(\Sigma_i)\epsilon,\delta_i)$, and proved the claim.

\end{proof}

Now we have a sequence of matching motions $(T_i,\mathcal {K}_i=\{\mu_{i,t}\})$, with $\mu_{i,0}=\mu_{\tilde\Sigma_i}$. Since $\tilde\Sigma_i$ are all closed hypersurfaces, the flow must become extinct in finite time $t_i$ (by the avoidance principle), at a round spherical singularity by Proposition \ref{thm11}.

Next we parabolically rescale the flows $(T_i,\mathcal K_i)$ by factors $\frac{1}{\sqrt{t_i}}$ and translate the extinction point $(x_{i},t_{i})$ to the space-time origin to get $(\tilde T_i,\mathcal {\tilde K}_i)$, $\mathcal{\tilde K}_i = D_{\frac{1}{\sqrt{t_i}}}(\mathcal K_i-(x_{i},t_{i}))=\{\tilde\mu_{i,t}\}$. The flows $\mathcal {\tilde K}_i$ has non-empty support and no spherical singularities for time $t\in[-1,0)$.

By Lemma \ref{thm20}, after throwing out small values of i, there is a $\tau\in(-1,-\frac{1}{2})$, independent of i, so that
\begin{gather}\label{eq7}
\mathrm{dist}_H(\spt(\tilde\mu_{i,-1}),\spt(\tilde\mu_{i,\tau}))+\mathrm{dist}_H(\sqrt{2n}\mathbb S^n,\sqrt{-2n\tau}\mathbb S^n)<\frac{1}{8}\epsilon
\end{gather}

By Brakke's compactness theorem, there is a subsequence $i(k)$ such that $\mathcal{\tilde K}_{i(k)}$ converges to a limit flow $\mathcal K_\infty=\{\mu_{\infty,t}\}_{t\in[-1,0]}$, which is a matching motion $(T_\infty,\mathcal K_\infty)$ by Theorem \ref{thm4}, and $\lambda(\mathcal K_\infty)=\Lambda_n$.

The uniqueness Theorem \ref{thm9} tells us this limit flow is the regular flow of round n-sphere.

Now we apply Proposition \ref{thm3} with $\epsilon=\frac{\tau+1}{2}$ and the limit flow being the regular round sphere. For sufficiently large $i(k)$, by connectedness of $\tilde\Sigma_i$, $\mathcal{\tilde K}_i$ is sufficiently close to the regular flow of sphere for $t\in(-1+\epsilon,0)$, namely we can choose $i(k_0)$ such that
\begin{gather}
\mathrm{dist}_H(\spt(\tilde\mu_{i(k_0),\tau}),\sqrt{-2n\tau}\mathbb S^n)<\frac{1}{4}\epsilon
\end{gather}

By (\ref{eq7}) and triangle inequality, we have
\begin{gather}
\mathrm{dist}_H(\spt(\tilde\mu_{i,-1}),\sqrt{2n}\mathbb S^n)<\frac{1}{8}\epsilon+\frac{1}{4}\epsilon<\frac{1}{2}\epsilon
\end{gather}

That is 
\begin{gather}
\mathrm{dist}_H(\tilde\Sigma_i,\sqrt{2nt_i}\mathbb S^n+\mathbf{x_i})<\frac{1}{2}\sqrt{t_i}\epsilon<\frac{1}{2}\sqrt{2nt_i}\epsilon
\end{gather}

contradicting our choice of $\tilde\Sigma_i$ in Claim \ref{thm23}, and proves the theorem.

\end{proof}

\section{Acknowledgment}

The author would like to express gratitude to his advisor Jacob Bernstein for generous support and helpful guidence.

\end{document}